\documentclass[reqno,11pt]{amsart}
\usepackage{amsmath,amssymb}
\usepackage{dsfont}         
\newtheorem{theorem}{Theorem}[section]

\newtheorem{lemma}[theorem]{Lemma}
\newtheorem{proposition}[theorem]{Proposition}
\theoremstyle{definition}

\theoremstyle{remark} \theoremstyle{remark}

\numberwithin{equation}{section}
\DeclareMathOperator*{\esssup}{\mathrm{ess\,sup}}

\allowdisplaybreaks[3]

\newcommand{\Rd}{\mathds{R}^d}
\newcommand{\intRd}{\int\limits_{\mathds{R}^d}}
\newcommand{\intGamN}{\int\limits_{\Gamma_0}}
\newcommand{\suml}[1]{\sum\limits_{#1}}
\newcommand{\prodl}[1]{\prod\limits_{#1}}
\newcommand{\intl}[1]{\int\limits_{#1}}
\newcommand{\bs}{\backslash}
\newcommand{\cOne}{\langle c_1 \rangle}
\newcommand{\cTwo}{\langle c_2 \rangle}
\newcommand{\phiOne}{\langle \phi_1 \rangle}
\newcommand{\phiTwo}{\langle \phi_2 \rangle}

\title[A kinetic equation for repulsive coalescing jumps]{A kinetic equation for repulsive coalescing random jumps in continuum }

\author{Krzysztof Pilorz}
\address{Instytut Matematyki, Uniwersytet Marii Curie-Sk{\l}odowskiej, 20-031 Lublin, Poland}
\email{kpilorz@gmail.com}

\begin{document}


\begin{abstract}
A continuum individual-based model of hopping and coalescing
particles is introduced and studied. Its microscopic dynamics are
described by a hierarchy of evolution equations obtained in the
paper. Then the passage from the micro- to mesoscopic dynamics is
performed by means of a Vlasov-type scaling. The existence and
uniqueness of the solutions of the corresponding kinetic equation
are proved.
\end{abstract}
\maketitle

\section{Introduction}
In this paper, we introduce and study the dynamics of an infinite
system of particles located in $\Rd$, which jump and merge
(coalesce). Both jumping and coalescing are repulsive. The proposed
model is individual based, which means that the description of its
Markov dynamics are performed in terms of random changes of states
of individual particles. In the proposed model, such changes
include: (a) the particle located at a given $x\in \Rd$ changes its
position to $y\in \Rd$ (jumping); (b) two particles, located at
$x\in \Rd$ and $y\in \Rd$, merge into a single particle located at
$z\in \Rd$ (coalescing). The rates of these events depend on the
configuration of all particles. The model proposed can be viewed as
an extension of the Kawasaki model with repulsion studied in
\cite{Kawasaki} where only random jumps with repulsion are taken
into account. To the best of our knowledge, the microscopic modeling
of this kind of merging is performed here for the first time. All
previous theories, see, e.g., \cite{Aldous,
EvolutionaryEquationsBanasiak, StrongFragmentation,
FragmentationCoagulation, Kolokoltsov, Lachowicz,
EvolutionaryEquationsLamb},  deal with the densities of the
particles and hence do not take into account the corpuscular
structure of the system. However,  we do not  take into account the
particle mass, which is supposed to be done in an extension of the
present model.

The state space of the system considered is the configuration space,
that is, the set of all locally finite subsets of $\Rd$
$$\Gamma = \Gamma(\Rd) = \Big\{ \gamma \subset \Rd: \gamma \cap \Lambda \text{ is finite for every compact } \Lambda \subset \Rd \Big\}.$$
It can be given a measurability structure which turns $\Gamma$ into
a standard Borel space. This allows one to consider probability
measures on $\Gamma$ as states of the system. To characterize them
one uses observables, which are appropriate functions $F: \Gamma
\rightarrow \mathds{R}$. For an observable $F$ and a state $\mu$,
the number
\[
\int_{\Gamma} F d \mu
\]
is the $\mu$-expected value of $F$. Then the evolution of states
$\mu_0 \mapsto \mu_t$ can be described via the dual evolution $F_0
\mapsto F_t$ based on the following duality relation
\[
\int_{\Gamma} F_0 d \mu_t =
\int_{\Gamma} F_t d \mu_0, \qquad t>0.
\]
The evolution of observables is obtained in turn from the Kolmogorov equation
\[
\frac{d}{dt} F_t = L F_t, \qquad F_t|_{t=0} = F_0,
\]
in which the `operator' $L$ characterizes the model, see
\cite{SLMandKE, VlasovScaling, StatisticalDynamics,
MarkovEvolutions} for more detail. In  the proposed model, it has
the following form
\begin{eqnarray}\label{OperatorL}
LF(\gamma) &=& \suml{\{x,y\} \subset \gamma} \  \intRd \tilde{c}_1(x,y;z;\gamma) \Big(F\big(\gamma \bs \{x,y\} \cup z \big) - F(\gamma) \Big) dz \nonumber \\[.2cm]
&+& \suml{x \in \gamma} \  \intRd \tilde{c}_2(x;y;\gamma) \Big(F\big(\gamma \bs x \cup y\big) - F(\gamma) \Big) dy.
\end{eqnarray}
The first term of $L$ describes the coalescence occurring with
intensity $\tilde{c}_1(x,y;z;\gamma)$. The particles located at $x$
and $y$  merge into a new particle located at a point $z$. The
second term describes the jump of the particle located at $x$  to a
point $y$ with intensity $\tilde{c}_2(x;y;\gamma)$. The model with
$L$ consisting of the second term only is the Kawasaki model studied
in \cite{Kawasaki}. The kernels $\tilde{c}_1$ and $\tilde{c}_2$ take
into account also the influence of the whole configuration, which is
supposed to be repulsive, see below.

In the present research, we follow the statistical approach, see,
e.g., \cite{Kawasaki,SLMandKE, VlasovScaling, StatisticalDynamics,
MarkovEvolutions}, in which the dynamics of the model are described
by means of that of the corresponding correlation functions obtained
from the following Cauchy problem
$$
\frac{d}{dt} k_t = L^\Delta k_t, \quad k_{t=0} = k_0,
$$
in which the `operator' $L^\Delta$ is related to $L$ in a certain
way. In section \ref{Passing}, we calculate $L^\Delta$ following the
scheme developed in \cite{MarkovEvolutions}. Usually, equations for
$k_t$ are studied in scales of the corresponding Banach spaces.
However, as the structure of $L^\Delta$ obtained below is too
complicated, in this work we do not study this equation, which is
supposed to be done in a separate work. Instead, in section
\ref{Vlasov}  we consider a simplified version obtained by means of
a Vlasov-type scaling procedure developed in, e.g.,
\cite{VlasovScaling}, which is equivalent to passing to the
so-called mesoscopic description. In particular, we informally
obtain the kinetic equation and study its local solutions in an
appropriate Banach space, showing their existence and uniqueness. In
the next section, we introduce necessary notions and technical
tools.

\section{Basic notions and tools}\label{Notions}

In this section we introduce basic notions and tools used for
proving the results in the following sections. We give only a short
description with references to the corresponding sources.

Note that each element $\gamma \in \Gamma$ is at most countable
without finite limiting points. $\Gamma$ is endowed with the vague
topology, which is the weakest topology that makes continuous the
mappings
$$\gamma \rightarrow \suml{x \in \gamma} f(x)$$
for all continuous compactly supported functions $f: \Rd \rightarrow
\mathds{R}$. For the general discussion on configuration spaces we
recommend \cite{SLMandKE, StatisticalDynamics}.

One can also consider the space of all finite configurations
$$\Gamma_0 = \Gamma_0(\Rd) = \Big\{ \eta \in \Gamma: \eta \text{ is finite} \Big\},$$
which can be written down as
$$\Gamma_0 = \bigsqcup\limits_{n=0}^\infty \Gamma^{(n)}$$
where $\Gamma^{(n)} = \Gamma^{(n)}( \Rd ) = \Big\{ \{    x_1, x_2,
... , x_n\} \subset \Rd, x_i \neq x_j \text{ for } i \neq j \Big\} ,
\ n \in \mathds{N}_0$ and each $\Gamma^{(n)}$ is equipped with the
topology based on the Euclidean topology of $\Rd$. Note that
$\Gamma^{(0)} = \{ \emptyset \}$. Therefore $\Gamma_0$ can be
considered either with the topology induced from the vague topology
of $\Gamma$ or with the topology of the disjoint union. These
topologies are different but the corresponding Borel
$\sigma-$algebras are equal. Both $(\Gamma, \mathcal{B} (\Gamma))$
and $(\Gamma_0, \mathcal{B} (\Gamma_0))$ are standard Borel spaces.
Additionally $\Gamma_0 \in \mathcal{B}(\Gamma)$, which implies that
$\mathcal{B}(\Gamma_0)$ is a sub-$\sigma$-field of
$\mathcal{B}(\Gamma)$.

For each  bounded Borel set $\Lambda \subset \Rd$, we define
$p_\Lambda(\gamma) = \gamma_\Lambda = \gamma \cap \Lambda$ and
denote $\Gamma_\Lambda = p_\Lambda(\Gamma)$. Note that
$\Gamma_\Lambda$ is a subset of $\Gamma_0$, that is $$\Gamma_\Lambda
= \bigsqcup\limits_{n=0}^\infty \Gamma_\Lambda^{(n)}, \quad
\text{where }\ \Gamma_\Lambda^{(n)} = \Gamma_\Lambda \cap
\Gamma^{(n)}.$$

We say that a Borel set $A \subset \Gamma_0$ is bounded, if for some $N \in \mathds{N}$ and a bounded $\Lambda \subset \Rd$
$$A \subset \bigsqcup\limits_{n=0}^N \Gamma_\Lambda^{(n)}.$$
It can be shown that  $G: \Gamma_0 \rightarrow \mathds{R}$ is
measurable if and only if, for any $n \in \mathds{N}_0$, there
exists a symmetric measurable function $G^{(n)}: (\Rd)^n \rightarrow
\mathds{R}$ such that $G^{(n)} (x_1, ..., x_n) = G(\{x_1, ..., x_n
\})$, whenever $x_i \neq x_j$ for $i \neq j$. Here we use the
convention that $G^{(0)} = G(\emptyset) \in \mathds{R}$ is just a
constant function.

By $B_{bs}(\Gamma_0)$ we denote the set of all bounded measurable
functions $G: \Gamma_0 \rightarrow \mathds{R}$ having bounded
supports. Then the $K$-transform is defined as follows. For any $G
\in B_{bs}(\Gamma_0)$, $KG: \Gamma \rightarrow \mathds{R}$ is
\begin{equation}\label{KDef}
(KG)(\gamma) = \suml{\eta \Subset \gamma} G(\eta),
\end{equation}
where $\eta \Subset \gamma$ means that $\eta$ is a finite
sub-configuration of $\gamma$,  see \cite{HarmonicAnalysis}.
Obviously, the $K-$transform is linear. It acts to
$\mathcal{F}_{cyl}(\Gamma)$, i.e., to the set of all measurable
cylinder functions $F: \Gamma \rightarrow \mathds{R}$. It is also
invertible with the inverse given by
\begin{equation*}
(K^{-1}F)(\eta) = \suml{\xi \subset \eta}(-1)^{|\eta \backslash \xi|}F(\xi).
\end{equation*}
For $G_1, G_2 \in B_{bs}(\Gamma_0)$, it is known that
\begin{equation}\label{KStar}
(KG_1)\cdot(KG_2) = K(G_1 \star G_2),
\end{equation}
where $G_1 \star G_2$ is the 'convolution' given by the formula
\begin{equation}\label{ConvDef}
(G_1 \star G_2)(\eta) = \sum_{\xi \subset \eta} G_1(\xi) \sum_{\zeta \subset \xi} G_2(\eta \backslash \xi \cup \zeta) \in B_{bs}(\Gamma_0).
\end{equation}
Denote
$$e(f, \gamma) = \prodl{x \in \gamma} f(x).$$
For a measurable compactly supported function $f: \Rd \rightarrow \mathds{R}$, the following holds
\begin{equation}\label{KProd}
K\big(e(f, \cdot)\big)(\gamma) = e(1+f, \gamma).
\end{equation}
A probability measure $\mu$  on $(\Gamma,\mathcal{B}(\Gamma))$ is
said to have finite local moments of all orders if for any $n \in
\mathds N$ and a bounded Borel $\Lambda \subset \Rd$,
$$ \intGamN |\gamma_\Lambda|^n \mu (d \gamma) < \infty,$$
where $|\eta|$ stands for the cardinality of $\eta \in \Gamma_0$.
For such a measure $\mu$,  one can define a correlation measure,
$\rho_\mu$, on $(\Gamma_0,\mathcal{B}(\Gamma_0))$ by
$$\intGamN G(\eta) \rho_\mu (d \eta) = \intl{\Gamma} (KG)(\gamma) \mu (d \gamma), \quad G \in B_{bs}(\Gamma_0).$$

By $\lambda$ we denote the Lebesgue-Poisson measure on
$(\Gamma_0,\mathcal{B}(\Gamma_0))$, that is, the correlation measure
for the homogeneous Poisson measure with unit intensity. The
Lebesgue-Poisson measure is uniquely defined by the following formula
\begin{equation}\label{LebPoisInt}
\intGamN G(\eta) \lambda(d\eta) = G^{(0)} + \suml{n=1}^\infty
\frac{1}{n!} \intl{(\Rd)^n} G^{(n)}(x_1, x_2, \dots, x_n) dx_1 dx_2
\cdots dx_n,
\end{equation}
which has to hold for all $G \in  B_{bs}(\Gamma_0)$. The Minlos
lemma (see, e.g.,  \cite[eq. (2.2)]{SLMandKE}), states that
\begin{eqnarray}\label{Minlos}
& \intGamN ... \intGamN G(\eta_1 \cup \eta_2 \cup ... \cup \eta_n) H(\eta_1, \eta_2, ..., \eta_n) \lambda(d\eta_1) \lambda(d\eta_2) ... \lambda(d\eta_n) \nonumber \\
& = \intGamN G(\eta) \sum H(\eta_1, \eta_2, ..., \eta_n) \lambda(d\eta),
\end{eqnarray}
where $n$ is a positive integer, $G:\Gamma_0 \rightarrow
\mathds{R}$, $H: (\Gamma_0)^n \rightarrow \mathds{R}$ are positive
and measurable and the sum is taken over all $n-$part partitions
$(\eta_1, ..., \eta_n)$ of $\eta$, where parts being empty
configurations are also considered.  For $n=2$, we can rewrite
(\ref{Minlos}) in the following form
\begin{equation}\label{MinlosSub}
\intGamN \intGamN G(\eta \cup \xi) H (\eta, \xi) \lambda(d\eta) \lambda(d\xi) = \intGamN G(\eta) \suml{\xi \subset \eta} H(\xi, \eta \bs \xi) \lambda(d\eta).
\end{equation}
By taking 
$$H(\eta_1, \eta_2) = \left\{ \begin{array}{ll} h(x, \eta_2), & \ \eta_1 = \{x\} \\ 0, & |\eta_1| \neq 1 \end{array} \right.$$
and using (\ref{LebPoisInt}) we obtain the following special case of the Minlos lemma
\begin{equation}\label{Minlos2}
\intGamN \intRd G(\eta \cup x) h(x,\eta) dx \lambda(d\eta) = \intGamN \suml{x \in \eta} G(\eta) h(x, \eta \bs x) \lambda(d \eta).
\end{equation}
Analogously, for 
$$H(\eta_1, \eta_2, \eta_3) = \left\{ \begin{array}{ll} h(x, y, \eta_3), & \ \eta_1 = \{x\}, \eta_2 = \{y\} \\ 0, & |\eta_1| \neq 1 \text{ or } |\eta_2| \neq 1, \end{array}  \right.$$
we have
\begin{eqnarray}\label{Minlos3}
& &\frac{1}{2} \intGamN \intRd \intRd G(\eta \cup \{x,y\}) h(x, y, \eta) dx dy \lambda(d\eta) \nonumber \\
& & \quad = \intGamN \suml{\{x, y\} \subset \eta} G(\eta) h(x, y, \eta \bs \{x, y\}) \lambda(d\eta)
\end{eqnarray}

\section{Dynamics of the correlation functions}\label{Passing}
In this section, we follow the approach of \cite{MarkovEvolutions} and obtain the operator $L^\Delta$  corresponding to (\ref{OperatorL}). First, by using the
$K-$transform (\ref{KDef}) we pass to the
quasi-observables and obtain the operator $\hat{L}$. Then by the
Minlos lemma (\ref{Minlos}) we obtain the operator $L^\Delta$.
Recall, that
$$L = L_1 + L_2,$$
where
$$L_1F(\gamma) = \sum_{\{x,y\} \subset \gamma} \  \intRd \tilde{c}_1(x,y;z;\gamma) \Big(F\big(\gamma \bs \{x,y\} \cup z \big) - F(\gamma) \Big) dz$$
and
$$L_2F(\gamma) = \sum_{x \in \gamma} \  \intRd \tilde{c}_2(x;y;\gamma) \Big(F\big(\gamma \backslash x \cup y\big) - F(\gamma) \Big) dy.$$
We assume that
\begin{eqnarray*}
\tilde{c}_1(x,y;z;\gamma) &=& c_1(x,y;z) e(t^{(1)}_z, \gamma \bs \{x,y\} ), \\
\tilde{c}_2(x;y;\gamma) &=& c_2(x;y) e(t^{(2)}_y, \gamma \bs x),
\end{eqnarray*}
with
$$t^{(1)}_z(u) = e^{-\phi_1(z-u)}, \ t^{(2)}_y(u) = e^{-\phi_2(y-u)}.$$
The above $c_1$, $c_2$, $\phi_1$, $\phi_2$ are positive real functions such that
\begin{eqnarray}\label{CoeffAsKTrans}
\tilde{c}_1(x,y;z;\gamma) &=& (KC^1_{x,y;z})(\gamma \backslash \{x,y\}), \nonumber \\
\tilde{c}_2(x;y;\gamma) &=& (KC^2_{x;y})(\gamma \backslash x)
\end{eqnarray}
for some $C^1_{x,y;z}$ and $C^2_{x;y}$. We discuss the form of these functions later in this section. Additionally, we assume that
\begin{eqnarray*}
& & c_1(x,y;z) = c_1(y,x;z), \\
& & \intl{(\Rd)^2} c_1(x_1,x_2;x_3) dx_i dx_j = \cOne \  < \infty, \ i, j = 1, 2, 3, \ i \neq j, \\
& & \intRd c_2(x;y) dx = \intRd c_2(x;y) dy = \cTwo \  < \infty, \\
& & \intRd \phi_1(x) dx = \phiOne \  < \infty, \quad \intRd \phi_2(x) dx = \phiTwo \  < \infty.
\end{eqnarray*}
Suppose that $F = KG$, where $G: \Gamma_0 \rightarrow \mathds{R}$. Then by writing $K\hat{L}G = LF$ we define
\begin{equation}\label{LHatDef}
\hat{L} = K^{-1}LK.
\end{equation}
By the properties of the $K-$transform we derive an explicit formula for $\hat L$.
\begin{proposition}\label{PropLHat}
$\hat{L}$ defined as above has the following form
\begin{eqnarray*}
\hat{L} G (\eta) &=& \int\limits_{\mathds{R}^d} \sum\limits_{\{x,y\} \subset \eta} \Big[C^1_{x,y;z} \star  H^1_{x,y;z}\Big] (\eta \backslash \{x,y\}) dz \nonumber \\
 &+& \intRd \sum\limits_{x \in \eta} \Big[C^2_{x;y} \star  H^2_{x;y}\Big] (\eta \backslash x) dy,
\end{eqnarray*}
where
\begin{eqnarray}\label{ConstH}
H^1_{x,y;z}(\eta) &=& G(\eta \cup z) - G(\eta \cup x) -G(\eta \cup y) -G(\eta \cup \{x,y\}), \nonumber \\
H^2_{x;y}(\eta) &=& G(\eta \cup y) - G(\eta \cup x).
\end{eqnarray}
\end{proposition}
The next step is to pass with the action of the operator $\hat{L}$ to the correlation functions, introducing $L^\Delta$. We can define the latter by the pairing $\langle\langle \hat{L}G, k \rangle\rangle = \langle\langle G, L^\Delta k \rangle\rangle$, that is
\begin{equation}\label{LTriDef}
\int\limits_{\Gamma_0} (\hat{L}G)(\eta) k(\eta) \lambda (d\eta) = \int\limits_{\Gamma_0} G(\eta) (L^\Delta k)(\eta) \lambda (d\eta)
\end{equation}
Let us consider the integral from the left hand side of equation (\ref{LTriDef}). Using the Minlos lemma (\ref{Minlos}) several times, we can transform it so that we can obtain $L^\Delta$.
\begin{proposition}\label{PropLTriangle}
$L^\Delta$ defined as above is of the form
$$L^{\Delta} = L^{\Delta}_1 + L^{\Delta}_2,$$
where
\begin{eqnarray*}
 L^{\Delta}_1 k(\eta) &=& \frac{1}{2} \int\limits_{(\mathds{R}^d)^2} \ \int\limits_{\Gamma_0} \sum\limits_{z \in \eta} c_1(x,y;z) k(\eta \backslash z \cup \xi \cup \{x,y\}) \\
 &\times & e(t^{(1)}_z - 1, \xi) e(t^{(1)}_z, \eta \bs z)\lambda(d\xi)  dx dy \\
&-& \frac{1}{2} \int\limits_{(\mathds{R}^d)^2} \ \int\limits_{\Gamma_0 } \sum\limits_{x \in \eta} c_1(x,y;z) k(\eta \cup \xi \cup y) \\
&\times & e(t^{(1)}_z - 1, \xi) e(t^{(1)}_z, \eta \bs x) \lambda(d\xi) dy dz \\
&-& \frac{1}{2} \int\limits_{(\mathds{R}^d)^2} \ \int\limits_{\Gamma_0 } \sum\limits_{y \in \eta} c_1(x,y;z) k(\eta \cup \xi \cup x) \\
&\times & e(t^{(1)}_z - 1, \xi) e(t^{(1)}_z, \eta \bs y) \lambda(d\xi) dx dz \\
&-& \int\limits_{\mathds{R}^d} \ \int\limits_{\Gamma_0 } \sum\limits_{\{x,y\} \subset \eta} c_1(x,y;z) k(\eta \cup \xi) \\
&\times & e(t^{(1)}_z - 1, \xi) e(t^{(1)}_z, \eta \bs \{x,y\})  \lambda(d\xi) dz
\end{eqnarray*}
and
\begin{eqnarray*}
L^{\Delta}_2 k(\eta) &=& \intRd \intGamN \sum\limits_{y \in \eta} k(\eta \backslash y \cup \xi \cup x) c_2(x;y) \\
&\times & e(t^{(2)}_y - 1, \xi) e(t^{(2)}_y - 1, \eta \bs y)  \lambda(d\xi) dx \\
&-& \intRd \intGamN k(\eta \cup \xi) \sum\limits_{x \in \eta} c_2(x;y) \\
&\times & \prodl{u \in \xi} e(t^{(2)}_y - 1, \xi) e(t^{(2)}_y - 1, \eta \bs x) \lambda(d\xi) dy.
\end{eqnarray*}
\end{proposition}
\begin{proof}[\bfseries{Proof of Proposition \ref{PropLHat}}]
First let us rewrite the operator $L$ in a more convenient form. Using (\ref{CoeffAsKTrans}) and recalling that any configuration treated as a subset of $R^d$ is Lebesgue measure-zero as it is countable, we have
\begin{eqnarray*}
L_1F(\gamma) &=& \sum_{\{x,y\} \subset \gamma} \  \int\limits_{\mathds{R}^d \backslash \gamma} \bigg(KC^1_{x,y;z}(\cdot)\Big[KG\big(\cdot \cup z\big) \\
 &&- KG\big(\cdot \cup \{x,y\}\big)\Big]\bigg)\big(\gamma \backslash \{ x,y\}\big)dz.
\end{eqnarray*}
Observe that for any $\xi \in \Gamma, \  x,y,z \notin \xi$ we have
$$KG(\xi \cup z) = \sum_{\eta \subset \subset \xi \cup z} G(\eta) = \sum_{\eta \subset \subset \xi}\Big[ G(\eta) + G(\eta \cup z)\Big] = K\Big[ G(\cdot) + G(\cdot \cup z)\Big](\xi)$$
and analogously
$$KG(\xi \cup \{x,y\}) = K \Big[ G(\cdot) + G(\cdot \cup x) + G(\cdot \cup y) + G(\cdot \cup \{ x,y \}) \Big](\xi).$$
Using linearity of the $K-$transform and above observations, we obtain
\begin{eqnarray*}
L_1F(\gamma) &=& \sum_{\{x,y\} \subset \gamma} \  \int\limits_{\mathds{R}^d} \bigg(KC^1_{x,y;z}(\cdot) K \Big[G(\cdot \cup z) - G(\cdot \cup x)  \\
 & & -G(\cdot \cup y) -G(\cdot \cup \{x,y\})\Big](\cdot)\bigg)(\gamma \backslash \{ x,y\})dz.
\end{eqnarray*}
Considering the second part of the operator, we have
\begin{eqnarray*}
L_2F(\gamma) &=& \sum\limits_{x \in \gamma} \  \intRd KC^2_{x;y}(\gamma \backslash x) \Big[KG(\gamma \backslash x \cup y) - KG(\gamma) \Big] dy \\
&=& \sum\limits_{x \in \gamma} \  \intRd \bigg( KC^2_{x;y}(\cdot) K \Big[G(\cdot \cup y) - G(\cdot \cup x) \Big] (\cdot) \bigg) (\gamma \backslash x) dy.
\end{eqnarray*}
Using notion (\ref{ConstH}) and property (\ref{KStar}) of the product of $K-$transforms, we derive
$$L_1F(\gamma) = \sum_{\{x,y\} \subset \gamma} \  \intRd \  K \Big[C^1_{x,y;z} \star  H^1_{x,y;z}\Big](\gamma \backslash \{ x,y\})dz,$$
$$L_2F(\gamma) = \sum\limits_{x \in \gamma} \  \intRd  K \Big[C^2_{x;y} \star  H^2_{x;y}\Big](\gamma \backslash x) dy.$$
Therefore
\begin{eqnarray*}
LF(\gamma) &=& \sum_{\{x,y\} \subset \gamma} \  \intRd \  K \Big[C^1_{x,y;z} \star  H^1_{x,y;z}\Big](\gamma \backslash \{ x,y\})dz \\
 &+& \ \sum\limits_{x \in \gamma} \  \intRd  K \Big[C^2_{x;y} \star  H^2_{x;y}\Big](\gamma \backslash x) dy.
\end{eqnarray*}
Recalling the definition (\ref{LHatDef}) of the operator $\hat{L}$ and denoting
$$\hat{L}_1 G(\eta) = K^{-1} L_1 F (\eta), \quad \hat{L}_2 G(\eta) = K^{-1} L_2 F (\eta)$$
we obtain
\begin{eqnarray*}
\hat{L}_1 G(\eta) &=& \sum\limits_{\xi \subset \eta} (-1)^{|\eta \backslash \xi|} \sum\limits_{\{x,y\} \subset \xi} \  \int\limits_{\mathds{R}^d} \  K \Big[C^1_{x,y;z} \star  H^1_{x,y;z}\Big](\xi \backslash \{ x,y\})dz \\
&=& \int\limits_{\mathds{R}^d} \sum\limits_{\{x,y\} \subset \eta} \ \sum\limits_{\xi \subset \eta \backslash {\{x,y\}}} (-1)^{|\eta \backslash {\{x,y\} \backslash \xi|}} K \Big[C^1_{x,y;z} \star  H^1_{x,y;z}\Big](\xi) dz \\
&=& \int\limits_{\mathds{R}^d} \sum\limits_{\{x,y\} \subset \eta} K^{-1} K \Big[C^1_{x,y;z} \star  H^1_{x,y;z}\Big] (\eta \backslash \{x,y\}) dz \\
&=& \int\limits_{\mathds{R}^d} \sum\limits_{\{x,y\} \subset \eta} \Big[C^1_{x,y;z} \star  H^1_{x,y;z}\Big] (\eta \backslash \{x,y\}) dz
\end{eqnarray*}
and analogously
\begin{eqnarray*}
\hat{L}_2 G(\eta) &=& \sum\limits_{\xi \subset \eta} (-1)^{|\eta \backslash \xi|} \sum\limits_{x \in \xi} \  \intRd  K \Big[C^2_{x;y} \star  H^2_{x;y}\Big](\xi \backslash x) dy \\
&=& \intRd \sum\limits_{x \in \eta} \ \sum\limits_{\xi \subset \eta \backslash x} (-1)^{|\eta \backslash x \backslash \xi|} K \Big[C^2_{x;y} \star  H^2_{x;y}\Big](\xi) dy \\
&=& \intRd \sum\limits_{x \in \eta} K^{-1} K \Big[C^2_{x;y} \star  H^2_{x;y}\Big] (\eta \backslash x) dy \\
&=& \intRd \sum\limits_{x \in \eta} \Big[C^2_{x;y} \star  H^2_{x;y}\Big] (\eta \backslash x) dy.
\end{eqnarray*}
Therefore
\begin{eqnarray*}
\hat{L} G (\eta) &=& \int\limits_{\mathds{R}^d} \sum\limits_{\{x,y\} \subset \eta} \Big[C^1_{x,y;z} \star  H^1_{x,y;z}\Big] (\eta \backslash \{x,y\}) dz \\
&+& \intRd \sum\limits_{x \in \eta} \Big[C^2_{x;y} \star  H^2_{x;y}\Big] (\eta \backslash x) dy.
\end{eqnarray*}
\end{proof}
\begin{proof}[\bfseries{Proof of Proposition \ref{PropLTriangle}}]
Using the special case (\ref{Minlos3}) of the Minlos \\ lemma and proposition (\ref{PropLHat}) we have
\begin{eqnarray*}
&& \int\limits_{\Gamma_0}(\hat{L}_1 G)(\eta) k(\eta) \lambda (d\eta) \\
&& =\int\limits_{\Gamma_0} \int\limits_{\mathds{R}^d} \sum\limits_{\{x,y\} \subset \eta} \Big[C^1_{x,y;z} \star  H^1_{x,y;z}\Big] (\eta \backslash \{x,y\}) k(\eta) dz \lambda (d\eta) \\
&& = \frac{1}{2} \int\limits_{(\mathds{R}^d)^3}  \int\limits_{\Gamma_0} \Big[C^1_{x,y;z} \star  H^1_{x,y;z}\Big] (\eta) k(\eta \cup \{x,y\}) \lambda(d\eta) dx dy dz.
\end{eqnarray*}
Recalling the definition (\ref{ConvDef}) of the convolution $\star$ and using the Minlos lemma in the form (\ref{MinlosSub}) twice, we obtain
\begin{eqnarray*}
\int\limits_{\Gamma_0}(\hat{L}_1 G)(\eta) k(\eta) \lambda (d\eta) &=& \frac{1}{2} \int\limits_{(\mathds{R}^d)^3} \intGamN \sum\limits_{\xi \subset \eta} C^1_{x,y;z}(\xi) \sum\limits_{\zeta \subset \xi} H^1_{x,y;z}(\eta \backslash \xi \cup \zeta) \\
&\times & k(\eta \cup \{x, y\}) \lambda(d\eta) dx dy dz \\
&=& \frac{1}{2} \int\limits_{(\mathds{R}^d)^3} \intGamN \intGamN C^1_{x,y;z}(\xi) \sum\limits_{\zeta \subset \xi} H^1_{x,y;z}(\eta \cup \zeta) \\
& \times & k(\eta \cup \xi \cup \{x, y\}) \lambda(d\eta) \lambda(d\xi) dx dy dz \\
&=& \frac{1}{2} \int\limits_{(\mathds{R}^d)^3}  \int\limits_{\Gamma_0} \int\limits_{\Gamma_0} \int\limits_{\Gamma_0} C^1_{x,y;z}(\xi \cup \zeta) H^1_{x,y;z}(\eta \cup \zeta) \\
& \times & k(\eta \cup \xi \cup \zeta \cup \{x,y\}) \lambda(d\eta) \lambda(d\xi) \lambda(d\zeta) dx dy dz.
\end{eqnarray*}
Using again the Minlos lemma (\ref{MinlosSub}), but in the opposite direction, we have
\begin{eqnarray*}
\int\limits_{\Gamma_0}(\hat{L}_1 G)(\eta) k(\eta) \lambda (d\eta) &=& \frac{1}{2} \int\limits_{(\mathds{R}^d)^3}  \int\limits_{\Gamma_0} \int\limits_{\Gamma_0} \sum\limits_{\zeta \subset \eta} C^1_{x,y;z}(\xi \cup \zeta) H^1_{x,y;z}(\eta) \\
&\times & k(\eta \cup \xi \cup \{x,y\}) \lambda(d\eta) \lambda(d\xi) dx dy dz \\
&=& \frac{1}{2} \int\limits_{(\mathds{R}^d)^3}  \int\limits_{\Gamma_0} H^1_{x,y;z}(\eta) \Big[ \int\limits_{\Gamma_0} k(\eta \cup \xi \cup \{x,y\})  \\
& \times & \sum\limits_{\zeta \subset \eta} C^1_{x,y;z}(\xi \cup \zeta) \lambda(d\xi) \Big] \lambda(d\eta) dx dy dz.
\end{eqnarray*}
Let us rewrite above using the definition (\ref{ConstH}) of $H^1_{x,y;z}(\eta)$.
\begin{eqnarray*}
&& \int\limits_{\Gamma_0}(\hat{L}_1 G)(\eta) k(\eta) \lambda (d\eta) \\
&& = \frac{1}{2} \int\limits_{(\mathds{R}^d)^3}  \int\limits_{\Gamma_0} \Big[ G(\eta \cup z) - G(\eta \cup x) -G(\eta \cup y) -G(\eta \cup \{x,y\})\Big] \\
&& \times \Big[ \int\limits_{\Gamma_0} k(\eta \cup \xi \cup \{x,y\}) \sum\limits_{\zeta \subset \eta} C^1_{x,y;z}(\xi \cup \zeta) \lambda(d\xi) \Big] \lambda(d\eta) dx dy dz
\end{eqnarray*}
Using the special cases (\ref{Minlos2}) and (\ref{Minlos3}) of the Minlos lemma we obtain
\begin{eqnarray*}
&& \int\limits_{\Gamma_0}(\hat{L}_1 G)(\eta) k(\eta) \lambda (d\eta) \\
&& = \intGamN G(\eta) \Big[ \frac{1}{2} \int\limits_{(\mathds{R}^d)^2} \int\limits_{\Gamma_0} \sum\limits_{z \in \eta} k(\eta \backslash z \cup \xi \cup \{x,y\}) \sum\limits_{\zeta \subset \eta \backslash z} C^1_{x,y;z}(\xi \cup \zeta) \lambda(d\xi)  dx dy \\
&& - \frac{1}{2} \int\limits_{(\mathds{R}^d)^2} \int\limits_{\Gamma_0} \sum\limits_{x \in \eta} k(\eta \cup \xi \cup y) \sum\limits_{\zeta \subset \eta \backslash x} C^1_{x,y;z}(\xi \cup \zeta) \lambda(d\xi) dy dz \\
&& - \frac{1}{2} \int\limits_{(\mathds{R}^d)^2}  \int\limits_{\Gamma_0} \sum\limits_{y \in \eta} k(\eta \cup \xi \cup x) \sum\limits_{\zeta \subset \eta \backslash y} C^1_{x,y;z}(\xi \cup \zeta) \lambda(d\xi) dx dz \\
&& - \int\limits_{\mathds{R}^d}  \int\limits_{\Gamma_0} \sum\limits_{\{x,y\} \subset \eta} k(\eta \cup \xi) \sum\limits_{\zeta \subset \eta \backslash \{x,y\}} C^1_{x,y;z}(\xi \cup \zeta) \lambda(d\xi) dz \Big] \lambda(d\eta).
\end{eqnarray*}
Employing the same technique to the second part of the operator $\hat{L}$, we derive
\begin{eqnarray*}
&&\int\limits_{\Gamma_0}(\hat{L}_2 G)(\eta) k(\eta) \lambda (d\eta) \\
&& = \int\limits_{(\mathds{R}^d)^2} \int\limits_{\Gamma_0} \int\limits_{\Gamma_0} \sum\limits_{\zeta \subset \eta} C^2_{x;y}(\xi \cup \zeta) H^2_{x;y}(\eta) k(\eta \cup \xi \cup x)  \lambda(d\eta) \lambda(d\xi) dx dy \\
&&= \intGamN G(\eta) \bigg[ \intRd \intGamN \sum\limits_{y \in \eta} k(\eta \backslash y \cup \xi \cup x) \sum\limits_{\zeta \subset \eta \backslash y} C^2_{x;y}(\xi \cup \zeta) \lambda(d\xi) dx \\
&&- \intRd \intGamN k(\eta \cup \xi) \sum\limits_{x \in \eta} \sum\limits_{\zeta \subset \eta \backslash x} C^2_{x;y}(\xi \cup \zeta) \lambda(d\xi) dy \bigg] \lambda(d\eta)
\end{eqnarray*}
Therefore, we obtain
\begin{eqnarray}\label{LTriRes}
L^{\Delta} k(\eta) &=& \ \ \frac{1}{2} \int\limits_{(\mathds{R}^d)^2} \int\limits_{\Gamma_0} \sum\limits_{z \in \eta} k(\eta \backslash z \cup \xi \cup \{x,y\}) \sum\limits_{\zeta \subset \eta \backslash z} C^1_{x,y;z}(\xi \cup \zeta) \lambda(d\xi)  dx dy \nonumber \\
&- & \frac{1}{2} \int\limits_{(\mathds{R}^d)^2} \int\limits_{\Gamma_0} \sum\limits_{x \in \eta} k(\eta \cup \xi \cup y) \sum\limits_{\zeta \subset \eta \backslash x} C^1_{x,y;z}(\xi \cup \zeta) \lambda(d\xi) dy dz \nonumber \\
& - & \frac{1}{2} \int\limits_{(\mathds{R}^d)^2} \int\limits_{\Gamma_0} \sum\limits_{y \in \eta} k(\eta \cup \xi \cup x) \sum\limits_{\zeta \subset \eta \backslash y} C^1_{x,y;z}(\xi \cup \zeta) \lambda(d\xi) dx dz \nonumber \\
& - & \int\limits_{\mathds{R}^d}  \int\limits_{\Gamma_0} \sum\limits_{\{x,y\} \subset \eta} k(\eta \cup \xi) \sum\limits_{\zeta \subset \eta \backslash \{x,y\}} C^1_{x,y;z}(\xi \cup \zeta) \lambda(d\xi) dz \nonumber \\
& + & \intRd \intGamN \sum\limits_{y \in \eta} k(\eta \backslash y \cup \xi \cup x) \sum\limits_{\zeta \subset \eta \backslash y} C^2_{x;y}(\xi \cup \zeta) \lambda(d\xi) dx \nonumber \\
& - & \intRd \intGamN \sum\limits_{x \in \eta} k(\eta \cup \xi) \sum\limits_{\zeta \subset \eta \backslash x} C^2_{x;y}(\xi \cup \zeta) \lambda(d\xi) dy
\end{eqnarray}
Note that so far we have not used any assumption about coefficients $\tilde{c}_1$ and $\tilde{c}_2$ but that they can be written as results of action of the $K-$transform on corresponding functions $C^1_{x,y;z}$ and $C^2_{x;y}$. Let us calculate explicit forms of these functions. Recall that
$$\tilde{c}_1(x,y;z;\gamma) = c_1(x,y;z) \prodl{u \in \gamma \backslash \{x,y\}} e^{-\phi_1(z-u)},$$
$$\tilde{c}_2(x;y;\gamma) = c_2(x;y) \prodl{u \in \gamma \backslash x} e^{-\phi_2(y-u)}.$$
We have
$$KC^1_{x,y;z} = c_1(x,y;z) e(t^{(1)}_z, \cdot),$$
that is
\begin{eqnarray*}
C^1_{x,y;z} &=& K^{-1} c_1(x,y;z)e(1 + t^{(1)}_z -1, \cdot) = c_1(x,y;z) K^{-1} \suml{\xi \subset \cdot} e(t^{(1)}_z - 1, \xi) \\
&=& c_1(x,y;z) K^{-1} K  e(t^{(1)}_z - 1, \cdot) = c_1(x,y;z) e(t^{(1)}_z - 1, \cdot).
\end{eqnarray*}
Therefore
\begin{equation}\label{C1Form}
C^1_{x,y;z}(\eta) = c_1(x,y;z) e(t^{(1)}_z - 1, \eta).
\end{equation}
Analogously we can derive
\begin{equation}\label{C2Form}
C^2_{x;y}(\eta) = c_2(x;y) e(t^{(2)}_y - 1, \eta).
\end{equation}
Using the above, we can rewrite the operator $L^\Delta$. For convenience let us denote the part of it corresponding to the coalescence, that is the first four terms of (\ref{LTriRes}), as $L^\Delta_1$ and the part corresponding to the jumps, that is the last two terms of (\ref{LTriRes}), as $L^\Delta_2$. Substituting (\ref{C1Form}) we derive
\begin{eqnarray*}
&&L^{\Delta}_1 k(\eta) = \frac{1}{2} \int\limits_{(\mathds{R}^d)^2} \int\limits_{\Gamma_0} \sum\limits_{z \in \eta} k(\eta \backslash z \cup \xi \cup \{x,y\}) \sum\limits_{\zeta \subset \eta \backslash z} c_1(x,y;z) \\
&& \times \   e(t^{(1)}_z - 1, \xi \cup \zeta)  \lambda(d\xi)  dx dy \\
&& -  \frac{1}{2} \int\limits_{(\mathds{R}^d)^2} \int\limits_{\Gamma_0} \sum\limits_{x \in \eta} k(\eta \cup \xi \cup y) \sum\limits_{\zeta \subset \eta \backslash x} c_1(x,y;z) e(t^{(1)}_z - 1, \xi \cup \zeta) \lambda(d\xi) dy dz \\
&& -  \frac{1}{2} \int\limits_{(\mathds{R}^d)^2} \int\limits_{\Gamma_0} \sum\limits_{y \in \eta} k(\eta \cup \xi \cup x) \sum\limits_{\zeta \subset \eta \backslash y} c_1(x,y;z) e(t^{(1)}_z - 1, \xi \cup \zeta) \lambda(d\xi) dx dz \\
&& -  \int\limits_{\mathds{R}^d}  \int\limits_{\Gamma_0} \sum\limits_{\{x,y\} \subset \eta} k(\eta \cup \xi) \sum\limits_{\zeta \subset \eta \backslash \{x,y\}} c_1(x,y;z) e(t^{(1)}_z - 1, \xi \cup \zeta) \lambda(d\xi) dz
\end{eqnarray*}
and analogously using (\ref{C2Form}) we obtain
\begin{eqnarray*}
L^{\Delta}_2 k(\eta) &=& \intRd \intGamN \sum\limits_{y \in \eta} k(\eta \backslash y \cup \xi \cup x) \sum\limits_{\zeta \subset \eta \backslash y} c_2(x;y) e(t^{(2)}_y - 1, \xi \cup \zeta)  \lambda(d\xi) dx \\
& -& \intRd \intGamN k(\eta \cup \xi) \sum\limits_{x \in \eta} \sum\limits_{\zeta \subset \eta \backslash x} c_2(x;y) e(t^{(2)}_y - 1, \xi \cup \zeta) \lambda(d\xi) dy.
\end{eqnarray*}
Consider the first component of $L^{\Delta}_1$ and denote it as
\begin{eqnarray*}
L^{\Delta}_{11} k(\eta) &=& \frac{1}{2} \int\limits_{(\mathds{R}^d)^2} \int\limits_{\Gamma_0} \sum\limits_{z \in \eta} k(\eta \backslash z \cup \xi \cup \{x,y\}) \\
& &\sum\limits_{\zeta \subset \eta \backslash z} c_1(x,y;z) e(t^{(1)}_z - 1, \xi \cup \zeta) \lambda(d\xi)  dx dy.
\end{eqnarray*}
Next, for a given $\eta$ let us introduce $C(\eta) = \{\xi \in \Gamma_0: \xi \cap \eta \neq \emptyset \}$. Then, because any configuration treated as a measurable subset of $\Rd$ is of Lebesgue measure 0 and the empty configuration does not belong to $C(\eta)$ for any $\eta \in\Gamma_0$, we have $\lambda(C(\eta)) = 0$ for every $\eta \in\Gamma_0$. Indeed, using the characterization (\ref{LebPoisInt}) of the integral w.r.t. the Lebesgue-Poisson measure we obtain
$$\lambda(C(\eta)) = \intGamN I_{C(\eta)}(\xi) \lambda(d\xi) = I_{C(\eta)}^{(0)} + \suml{n=1}^\infty\frac{1}{n!} \intl{(\Rd)^n} I_{C(\eta)}^{(n)}(x_1, ..., x_n) dx_1 ... dx_n.$$
First, notice that $I_{C(\eta)}^{(0)} = 0$, as empty configuration cannot have common part with any configuration. Then, because
$$I_{C(\eta)}^{(n)}(x_1, ..., x_n) \leq I_{C(\eta)}^{(1)}(x_1) + I_{C(\eta)}^{(1)}(x_2) + ... + I_{C(\eta)}^{(1)}(x_n)$$
 we have for every $n \in \mathds{N}$
$$\intl{(\Rd)^n} I_{C(\eta)}^{(n)}(x_1, ..., x_n) dx_1 ... dx_n \leq n \intl{(\Rd)^{n-1}} \Big[ \intRd I_{C(\eta)}^{(1)}(x) dx \Big] dx_1...dx_{n-1}.$$
Taking into account that
$$\intRd I_{C(\eta)}^{(1)}(x) dx = \intRd I_\eta(x)dx = l(\eta) = 0,$$
where $l$ denotes the Lebesgue measure, one can clearly see that $\lambda(C(\eta)) = 0$. \\
Therefore, when integrating over $\Gamma_0 \bs C(\eta)$ instead of $\Gamma_0$, the result is the same. However, all subconfigurations $\zeta$ of $\eta$ are disjoint with any $\xi \in \Gamma_0 \bs C(\eta)$, which allows us to separate the product taken over $\xi \cup \zeta$ into one taken over $\xi$ and another taken over $\zeta$. Thus we can write
\begin{eqnarray*}
L^{\Delta}_{11} k(\eta) &=& \frac{1}{2} \int\limits_{(\mathds{R}^d)^2} \intGamN \sum\limits_{z \in \eta} c_1(x,y;z) k(\eta \backslash z \cup \xi \cup \{x,y\}) \\
&\times & e(t^{(1)}_z - 1 , \xi)  \sum\limits_{\zeta \subset \eta \backslash z} \  e(t^{(1)}_z - 1 , \zeta) \lambda(d\xi) dx dy.
\end{eqnarray*}
Recalling the definition (\ref{KDef}) of the $K-$transform and its property (\ref{KProd}) we have
$$\sum\limits_{\zeta \subset \eta \backslash z}e(t^{(1)}_z - 1 , \zeta) = K\big(e(t^{(1)}_z - 1 , \cdot)\big) (\eta \backslash z) = e(t^{(1)}_z , \eta \bs z).$$
Therefore we can rewrite the action of $L^{\Delta}_{11}$ in the form
\begin{eqnarray*}
L^{\Delta}_{11} k(\eta) &=& \frac{1}{2} \int\limits_{(\mathds{R}^d)^2} \ \int\limits_{\Gamma_0} \sum\limits_{z \in \eta} c_1(x,y;z) k(\eta \backslash z \cup \xi \cup \{x,y\}) \\
&\times & e(t^{(1)}_z - 1 , \xi) e(t^{(1)}_z , \eta \bs z) \lambda(d\xi)  dx dy.
\end{eqnarray*}
Applying the same method for the rest of the $L^\Delta_1$ and for the $L^\Delta_2$ we obtain the result.
\end{proof}

\section{The Vlasov Scaling and the Kinetic Equation}\label{Vlasov}
\subsection{The Vlasov Scaling}~\\
We follow the scaling technique described in \cite{VlasovScaling}. Let us introduce the scale parameter $\epsilon \in [0,1]$ with $\epsilon = 1$ corresponding to the unscaled and $\epsilon =0$ to the fully rescaled case. We alter the operator $L^\Delta$ by scaling $c_1 \rightarrow \epsilon c_1$, $\phi_1 \rightarrow \epsilon c_2$ and $\phi_2 \rightarrow \epsilon \phi_2$ for $\epsilon \in (0,1]$, which can be interpreted as weakening the interactions between particles. Altered in such a way operator we denote by $L^{\Delta}_\epsilon$. Next, we renormalize it defining
\begin{equation*}
L^{ren}_{\epsilon} k(\eta) = \epsilon^{|\eta|} L^{\Delta}_\epsilon(\epsilon^{-|\eta|} k(\eta)).
\end{equation*}
Let us consider the first component of the operator $L^{ren}_{\epsilon}$. We have (cf Proposition \ref{PropLTriangle})
\begin{eqnarray*}
L^{ren}_{11, \epsilon} k (\eta) &=& \frac{1}{2} \epsilon^{|\eta|} \int\limits_{(\mathds{R}^d)^2} \ \int\limits_{\Gamma_0} \suml{z \in \eta} \epsilon c_1(x,y;z) \epsilon^{-|\eta \backslash z \cup \xi \cup \{x,y\}|} k(\eta \backslash z \cup \xi \cup \{x,y\}) \\
&\times & \prodl{u \in \xi} (e^{-\epsilon\phi_1(z-u)} - 1) \prodl{u \in \eta \backslash z} e^{-\epsilon\phi_1(z-u)} \lambda(d\xi)  dx dy.
\end{eqnarray*}
Note that integrating over $(\Rd)^2 \bs (\eta \times \eta)$ instead of $(\Rd)^2$ and over  $\Gamma_0 \bs (\eta \cup \{x,y\})$ instead of $\Gamma_0$ does not influence the result, so we have
\begin{eqnarray*}
L^{ren}_{11, \epsilon} k (\eta) &=& \frac{1}{2} \int\limits_{(\mathds{R}^d)^2} \ \int\limits_{\Gamma_0} \suml{z \in \eta} c_1(x,y;z) k(\eta \backslash z \cup \xi \cup \{x,y\}) \\
&\times & \prodl{u \in \xi} \frac{1}{\epsilon} \Big( e^{-\epsilon\phi_1(z-u)} - 1 \Big) \prodl{u \in \eta \backslash z} e^{-\epsilon\phi_1(z-u)} \lambda(d\xi)  dx dy.
\end{eqnarray*}
Let us pass with $\epsilon$ to the limit. Noting that
$$\lim\limits_{\epsilon \rightarrow 0} \frac{1}{\epsilon} \Big( e^{-\epsilon\phi_1(z-u)} - 1 \Big) = -\epsilon\phi_1(z-u),$$
we can write
\begin{eqnarray*}
\lim\limits_{\epsilon \rightarrow 0} L^{ren}_{11, \epsilon} k (\eta) &=& \frac{1}{2} \int\limits_{(\mathds{R}^d)^2} \ \int\limits_{\Gamma_0} \suml{z \in \eta} c_1(x,y;z) k(\eta \backslash z \cup \xi \cup \{x,y\}) \\
& \times & \prodl{u \in \xi} \big(-\phi_1(z-u) \big)\lambda(d\xi) dx dy.
\end{eqnarray*}
Let us denote $V = \lim\limits_{\epsilon \rightarrow 0} L^{ren}_{\epsilon}$. Calculating analogously as above, one derives
\begin{eqnarray*}
V k(\eta) & = & \frac{1}{2} \int\limits_{(\mathds{R}^d)^2} \ \int\limits_{\Gamma_0} \suml{z \in \eta} c_1(x,y;z) k(\eta \backslash z \cup \xi \cup \{x,y\}) \nonumber \\
& \times & \prodl{u \in \xi} (-\phi_1(z-u)  )\lambda(d\xi) dx dy \nonumber \\
& - &\frac{1}{2} \int\limits_{(\mathds{R}^d)^2} \ \int\limits_{\Gamma_0} \suml{x \in \eta} c_1(x,y;z) k(\eta \cup \xi \cup y) \prodl{u \in \xi} (-\phi_1(z-u)   )\lambda(d\xi) dy dz \nonumber \\
& - &\frac{1}{2} \int\limits_{(\mathds{R}^d)^2} \ \int\limits_{\Gamma_0} \suml{y \in \eta} c_1(x,y;z) k(\eta \cup \xi \cup x) \prodl{u \in \xi} (-\phi_1(z-u)   )\lambda(d\xi) dx dz \nonumber \\
& + &\intRd \intGamN \sum\limits_{y \in \eta} k(\eta \backslash y \cup \xi \cup x) c_2(x;y) \prodl{u \in \xi} (-\phi_2(y-u)) \lambda(d\xi) dx \nonumber \\
& - & \intRd \intGamN k(\eta \cup \xi) \sum\limits_{x \in \eta} c_2(x;y) \prodl{u \in \xi} (-\phi_2(y-u)) \lambda(d\xi) dy.
\end{eqnarray*}
Consider the following problem
\begin{equation}\label{VlasovEq}
\frac{d}{dt} r_t = V r_t, \quad r_{t=0} = r_0
\end{equation}
in the Banach space
$$\mathcal{K}_\theta = \{r: \Gamma_0 \rightarrow \Rd: ||r||_\theta < \infty \},$$
where
$$||r||_\theta = \esssup_{\eta \in \Gamma_0} e^{\theta |\eta|} ||r(\eta)||_{L^\infty} $$
Supposing that the initial state is the Poisson measure, $r_0$ can be factorized
$$r_0(\eta) = \prodl{x \in \eta} \rho_0(x).$$
If $r_t$ can be written in the product form, that is
\begin{equation*}
r_t(\eta) = \prodl{x \in \eta} \rho_t(x),
\end{equation*}
then we can write
$$\frac{d}{dt} r_t (\eta) = \frac{d}{dt} \prodl{x \in \eta} \rho_t(x) = \suml{x \in \eta} \Big( \prodl{y \in \eta \bs x} \rho_t(y) \Big) \frac{d}{dt} \rho_t(x). $$
Therefore, by expressing $V r_t$ in the form
$$V r_t (\eta) = \suml{x \in \eta} \Big( \prodl{y \in \eta \bs x} \rho_t(y) \Big) v(\rho_t, x),$$
we can obtain a problem for $\rho_t$ corresponding to (\ref{VlasovEq}), namely a kinetic equation
\begin{equation}\label{KineticGeneral}
\frac{d}{dt} \rho_t (x) = v(\rho_t, x), \quad \rho_{t=0} = \rho_0,
\end{equation}
where $r_0(x) = \prodl{x \in \eta} \rho_0(x)$.
Indeed, if $\rho_t$ is a solution of (\ref{KineticGeneral}), then we can easily check that
$$k_t(\eta) = \prodl{x \in \eta} \rho_t(x)$$
is a solution of (\ref{VlasovEq}).

Let us denote the first component of $V$ by $V_1$. We have
\begin{eqnarray*}
V_1 r_t (\eta) &=& \frac{1}{2} \int\limits_{(\mathds{R}^d)^2} \ \int\limits_{\Gamma_0} \suml{z \in \eta} c_1(x,y;z) r_t(\eta \backslash z \cup \xi \cup \{x,y\}) \\
&\times & \prodl{u \in \xi} (-\phi_1(z-u))\lambda(d\xi) dx dy \\
&=& \frac{1}{2} \int\limits_{(\mathds{R}^d)^2} \ \int\limits_{\Gamma_0} \suml{z \in \eta} c_1(x,y;z) \prodl{v \in \eta \backslash z \cup \xi \cup \{x,y\}} \rho_t(v) \\
& \times & \prodl{u \in \xi} (-\phi_1(z-u))\lambda(d\xi) dx dy.
\end{eqnarray*}
Because $l(\eta) = 0$ and $\lambda(C(\eta \cup \{x,y\})) = 0$, one can rewrite above as
\begin{eqnarray*}
&&V_1 r_t (\eta) = \suml{z \in \eta} \Big( \prodl{v \in \eta \bs z} \rho_t(v) \Big) \Big( \frac{1}{2} \int\limits_{(\mathds{R}^d)^2 \bs (\eta \times \eta)} \ \int\limits_{\Gamma_0 \bs C(\eta \cup \{x,y\})} c_1(x,y;z) \\
&& \times \ \rho_t(x) \rho_t(y) \prodl{u \in \xi} (-\rho_t(u) \phi_1(z-u))\lambda(d\xi) dx dy \Big).
\end{eqnarray*}
Therefore
$$V_1 r_t (\eta) = \suml{z \in \eta} \Big( \prodl{v \in \eta \bs z} \rho_t(v) \Big) v_1(\rho_t, z),$$
where
$$v_1(\rho_t, z) = \frac{1}{2} \int\limits_{(\mathds{R}^d)^2} c_1(x,y;z) \rho_t(x) \rho_t(y) \int\limits_{\Gamma_0} \prodl{u \in \xi} (-\rho_t(u) \phi_1(z-u))\lambda(d\xi) dx dy.$$
Noting that for $ a(u) = -\rho_t(u) \phi_1(z-u)$
\begin{eqnarray*}
\int\limits_{\Gamma_0} \prodl{u \in \xi} a(u) &=& 1 + \suml{n=1}^\infty \frac{1}{n!} \intl{(\Rd)^n} a(x_1) ... a(x_n) dx_1 ... dx_n \\
&=& 1 + \prodl{n=1}^\infty \frac{1}{n!} \Big( \intRd a(u) du \Big)^n = \exp \Big( \intRd a(u) du \Big),
\end{eqnarray*}
we can reformulate the above obtaining
$$v_1(\rho_t, x) = \frac{1}{2} \int\limits_{(\mathds{R}^d)^2} c_1(y,z;x) \rho_t(y) \rho_t(z) \exp \Big( - \intRd \phi_1(x-u) \rho_t(u) du \Big) dy dz. $$
Calculating analogously as above, one can obtain explicit form of $v$ and thus the following kinetic equation
\begin{eqnarray}\label{KinEq}
 \frac{d}{dt} \rho_t (x) & =&  \frac{1}{2} \int\limits_{(\mathds{R}^d)^2} c_1(y,z;x) \exp \Big( - \intRd \phi_1(x-u) \rho_t(u) du \Big) \rho_t(y) \rho_t(z) dy dz \nonumber \\
 & -& \frac{1}{2} \int\limits_{(\mathds{R}^d)^2} \Big( c_1(x,y;z) + c_1(y,x;z) \Big) \nonumber \\
 & \times & \exp \Big( - \intRd \phi_1(z-u) \rho_t(u) du \Big) \rho_t(x) \rho_t(y) dy dz \nonumber \\
 & + & \intRd  c_2(y;x) \exp \Big( - \intRd \phi_2(x-u) \rho_t(u) du \Big) \rho_t(y) dy \nonumber \\
 & - & \intRd  c_2(x;y) \exp \Big( - \intRd \phi_2(y-u) \rho_t(u) du \Big) \rho_t(x) dy, \nonumber \\
 \rho_{t=0} &=& \rho_0.
\end{eqnarray}

\subsection{The Kinetic Equation}~\\
Let us rewrite the problem (\ref{KinEq}) as
\begin{equation}\label{KinEq2}
\frac{d}{dt} \rho_t (x) = R_1(\rho_t,x) + R_2(\rho_t,x), \quad \rho_{t=0}(x) = \rho_0 (x),
\end{equation}
where
\begin{eqnarray*}
R_1(\rho_t,x) &=& - \frac{1}{2} \rho_t(x) \intl{(\Rd)^2} \Big( c_1(x,y;z) + c_1(y,x;z) \Big) \rho_t(y) dy dz \\
&-& h(\rho_t, x) \intRd c_2(x;y) dy \\
&=& - \rho_t (x) h(\rho_t, x)
\end{eqnarray*}
for
$$h(\rho_t, x) = \frac{1}{2} \intl{(\Rd)^2} \Big( c_1(x,y;z) + c_1(y,x;z) \Big) \rho_t(y) dy dz \ + \cTwo $$
and
\begin{eqnarray*}
R_2(\rho_t, x) & =& \frac{1}{2} \int\limits_{(\mathds{R}^d)^2} c_1(y,z;x) \exp \Big( - \intRd \phi_1(x-u) \rho_t(u) du \Big) \rho_t(y) \rho_t(z) dy dz \\
 & +& \frac{1}{2} \int\limits_{(\mathds{R}^d)^2} \Big( c_1(x,y;z) + c_1(y,x;z) \Big) \\
 & \times & \Big[1 - \exp \Big( - \intRd \phi_1(z-u) \rho_t(u) du \Big) \Big] \rho_t(x) \rho_t(y) dy dz \\
 & +& \intRd  c_2(y;x) \exp \Big( - \intRd \phi_2(x-u) \rho_t(u) du \Big) \rho_t(y) dy \\
 & +& \intRd  c_2(x;y) \Big[1 - \exp \Big( - \intRd \phi_2(y-u) \rho_t(u) du \Big) \Big] \rho_t(x) dy.
\end{eqnarray*}
Note that from (\ref{KinEq2}) we can obtain the equivalent integral equation
\begin{equation}\label{IntEq}
\rho_t(x) = \rho_0(x) \exp \Big( -\intl{0}^t h(\rho_s, x) ds \Big) + \intl{0}^t R_2(\rho_s,x) \exp \Big( - \intl{s}^t h(\rho_\sigma, x) d\sigma \Big) ds.
\end{equation}
\begin{theorem}\label{Thm}
Problem (\ref{KinEq2}) with the initial condition $\rho_0 \in L^\infty(\Rd)), \\ \rho_0 \geq 0$ has the unique local classical solution.
\end{theorem}
\noindent Consider $X_T = C([0, T] \rightarrow L^\infty(\Rd)), \ T > 0$ with the norm
$$||\rho||_{T,\gamma} = \sup\limits_{t \in [0,T]} e^{- \gamma \cTwo t} ||\rho_t||_{L^\infty} . $$
Denote
\begin{eqnarray*}
&&B_{T,\gamma}(r) = \{ \rho \in X_T: \ ||\rho||_{T,\gamma} \leq r,  \ \rho_t \geq 0 \  \forall t \in [0,T]\},\\
&&B_{T,\gamma}(r, \rho_0) = \{ \psi \in B_{T,\gamma}(r): \psi_0 = \rho_0 \},
\end{eqnarray*}
where $\rho_0 \in L^\infty(\Rd), \rho_0 \geq 0$, $r \geq ||\rho_0||_{L^\infty}$ and $T, \gamma > 0$.
\begin{lemma}\label{LemmaOne}
Given $r > 0$, there exist $\gamma, \tilde{T} > 0$ such that $F$ defined by the RHS of (\ref{IntEq}) with the domain $B_{T^*,\gamma}(r) \subset X_{T^*}$ acts again to the $B_{T^*,\gamma}(r)$ for any $T^* \in [0, \tilde{T}]$.
\end{lemma}
\begin{lemma}\label{LemmaTwo}
Let $\rho_0 \in L^\infty(\Rd), \rho_0 \geq 0$ and $r \geq ||\rho_0||_{L^\infty}$. Let $\tilde{T}, \gamma$ satisfy lemma \ref{LemmaOne} for this $r$. We can choose $T^* \in [0, \tilde{T}]$ in such a way that for any  $\rho$, $\psi$ in $B_{T^*,\gamma}(r, \rho_0)$ the inequality $||F(\rho) - F(\psi)||_{T^*, \gamma} \leq C ||\rho - \psi||_{T^*, \gamma}$ holds for some constant $C < 1$.
\end{lemma}
\begin{proof}[\bfseries{Proof of Theorem \ref{Thm}}]
Choose $r > ||\rho_0||_{L^\infty}$ and take corresponding $\gamma, \tilde{T}$ from lemma \ref{LemmaOne}. Take $T^*$ as in lemma \ref{LemmaTwo}. Define the sequence of Picard iterations $(\rho^{(n)})_{n \in \mathds{N}_0}$ in the following way
\begin{eqnarray}\label{ThmSeq}
&&\rho^{(0)}_t = \rho_0 \ \forall t \in [0, T^*], \nonumber \\
&&\rho^{(n)} = F(\rho^{(n-1)}), \ n \in \mathds{N}.
\end{eqnarray}
Obviously, $\rho^{(0)} \in B_{T^*, \gamma}(r)$. Therefore, by lemma \ref{LemmaOne}, $\rho^{(n)} \in B_{T^*, \gamma}(r)$ for all $n \in \mathds{N_0}$ and from lemma \ref{LemmaTwo} we obtain
$$||\rho^{(n+k)} - \rho^{(n)} ||_{T^*,\gamma} \leq ||\rho^{(1)} - \rho^{(0)} ||_{T^*,\gamma} \suml{i=1}^k C^{n+i-1} \leq ||\rho^{(1)} - \rho^{(0)} ||_{T^*,\gamma} \frac{C^n}{1-C},$$ where $C < 1$ is a positive constant. Therefore $\big( \rho^{(n)} \big)_{n \in \mathds{N}_0}$ defined by (\ref{ThmSeq}) is a Cauchy sequence. As $B_{T^*, \gamma}(r)$ is a closed subset of a Banach space, there exists $$\lim\limits_{n \rightarrow \infty} \rho^{(n)} = \rho \in B_{T^*, \gamma}(r). $$
Clearly $F(\rho) = \rho$ and therefore $\rho_t $ satisfies the integral equation (\ref{IntEq}) for $t \in [0, T^*]$. Thus it is a local classical solution of (\ref{KinEq2}). \\
Now suppose there is another local classical solution of this equation $\psi$. Then $\psi_0 = \rho_0$ and for $r, \gamma, T^*$ as above, there exists $T \leq T^*$ such that $\psi \in B_{T, \gamma}(r)$. However, from lemma \ref{LemmaTwo} we have
$$||\rho - \psi||_{T, \gamma} = ||F(\rho) - F(\psi)||_{T, \gamma} \leq C ||\rho - \psi||_{T, \gamma}$$
for $C < 1$, which means that $$||\rho - \psi||_{T, \gamma} = 0 $$ and thus $\rho$ is the unieque local classical solution.
\end{proof}
\begin{proof}[\bfseries{Proof of Lemma \ref{LemmaOne}}]
Take arbitrary $T, \gamma > 0$ and $\rho \in B_{T, \gamma}(r)$. Note that
\begin{eqnarray}\label{LemmaOneEq}
&& h(\rho_t, x) \geq \ \cTwo, \nonumber \\
&&R_2(\rho_t,x) \leq \frac{3}{2} ||\rho_t||^2_{L^\infty} \cOne + 2 ||\rho_t||_{L^\infty} \cTwo, \nonumber \\
&&\rho_t(x) \leq ||\rho_t||_{L^\infty} \leq e^{\gamma \cTwo t} ||\rho||_{T, \gamma}.
\end{eqnarray}
It is obvious that $F$ preserves positiveness of $\rho$. Furthermore, using above estimates and the definition of $B_{T,\gamma}(r)$ we derive
\begin{eqnarray*}
\Big( F(\rho) \Big)_t(x) &=& \rho_0(x) \exp \Big( -\intl{0}^t h(\rho_s, x) ds \Big) \\
&+& \intl{0}^t R_2(\rho_s,x) \exp \Big( - \intl{s}^t h(\rho_\sigma, x) d\sigma \Big) ds \\
&\leq & ||\rho_0||_{L^\infty} e^{-t \cTwo} + \intl{0}^t R_2(\rho_s,x) e^{(s-t) \cTwo} ds \\
&\leq & e^{-t \cTwo} \Big[ ||\rho||_{T,\gamma} + \intl{0}^t \Big( \frac{3}{2} \cOne e^{(2 \gamma + 1) \cTwo s} ||\rho||^2_{T,\gamma} \\
&+& 2 \cTwo e^{(\gamma + 1) \cTwo s} ||\rho||_{T,\gamma} \Big) ds  \Big].
\end{eqnarray*}
Therefore we obtain
\begin{eqnarray*}
\Big|\Big|\Big( F(\rho) \Big)_t\Big|\Big|_{L^\infty} &\leq & e^{-t \cTwo} r \Big[1 + \frac{3 \cOne r }{2 (2 \gamma + 1) \cTwo} \Big( e^{(2 \gamma + 1) \cTwo t} - 1 \Big) \\
& +& \frac{2}{\gamma + 1} \Big( e^{(\gamma + 1) \cTwo t} - 1 \Big) \Big].
\end{eqnarray*}
Thus
$$\Big|\Big|\Big( F(\rho) \Big)\Big|\Big|_{T, \gamma} \leq r \sup\limits_{t \in [0,T]} f(t), $$
where
\begin{eqnarray*}
f(t) &=& e^{-(\gamma+1) \cTwo t} \Big[1 + \frac{3 \cOne r }{2 (2 \gamma + 1) \cTwo} \Big( e^{(2 \gamma + 1) \cTwo t} - 1 \Big) \\
 &+& \frac{2}{\gamma + 1} \Big( e^{(\gamma + 1) \cTwo t} - 1 \Big) \Big].
\end{eqnarray*}
Note that $f(0) = 1$. Additionally
\begin{eqnarray*}
f'(t) &=& -(\gamma + 1) \cTwo e^{-(\gamma+1) \cTwo t} \Big[1 + \frac{3 \cOne r }{2 (2 \gamma + 1) \cTwo} \Big( e^{(2 \gamma + 1) \cTwo t} - 1 \Big) \\
 &+& \frac{2}{(\gamma + 1)} \Big( e^{(\gamma + 1) \cTwo t} - 1 \Big) \Big] \\
  &+& e^{-(\gamma+1) \cTwo t} \Big[\frac{3 \cOne r }{2} e^{(2 \gamma + 1) \cTwo t}  + 2 \cTwo e^{(\gamma + 1) \cTwo t} \Big].
\end{eqnarray*}
and hence
$$f'(0) = -(\gamma + 1)\cTwo + \Big(\frac{3}{2} \cOne r +2 \cTwo \Big).$$
Choosing $\gamma > 1 + \frac{3 \cOne r}{2 \cTwo}$ we have $f'(0) < 0$, which guarantees existence of $\tilde{T}$ such that $\sup\limits_{t \in [0,\tilde{T}]} f(t) = 1$. Taking $T = T^*$ for $T^* \in [0,\tilde{T}]$ yields
$$ ||F(\rho)||_{T^*, \gamma} \leq r.$$
Therefore $F(\rho) \in B_{T^*, \gamma}(r)$ for $\rho \in B_{T^*, \gamma}(r)$.
\end{proof}
\begin{proof}[\bfseries{Proof of lemma \ref{LemmaTwo}}]
We have
\begin{eqnarray}\label{LemmaTwoEq}
&&\Big(F(\rho) - F(\psi)\Big)_t (x) \nonumber \\
 && = \rho_0(x) \exp \Big( -\intl{0}^t h(\rho_s, x) ds \Big) + \intl{0}^t R_2(\rho_s,x) \exp \Big( - \intl{s}^t h(\rho_\sigma, x) d\sigma \Big) ds \nonumber \\
&& - \rho_0(x) \exp \Big( -\intl{0}^t h(\psi_s, x) ds \Big) - \intl{0}^t R_2(\psi_s,x) \exp \Big( - \intl{s}^t h(\psi_\sigma, x) d\sigma \Big) ds \nonumber \\
&& = D_1 + \intl{0}^t D_2 ds,
\end{eqnarray}
where
$$D_1 = \rho_0(x) \Big[ \exp\Big( - \intl{0}^t h(\rho_s, x) ds  \Big) - \exp\Big( - \intl{0}^t h(\psi_s, x) ds  \Big)  \Big]$$
and
\begin{eqnarray*}
D_2 &=& \intl{0}^t \Big[ R_2(\rho_s,x) \exp \Big( - \intl{s}^t h(\rho_\sigma, x) d\sigma \Big) \\
&-& R_2(\psi_s,x) \exp \Big( - \intl{s}^t h(\psi_\sigma, x) d\sigma \Big)  \Big]ds.
\end{eqnarray*}
Take an arbitrary $T^* \in [0, \tilde{T}]$. We have
\begin{equation}\label{LemmaTwoEstimateZero}
|D_1| \leq ||\rho_0||_{L^\infty} \cOne \intl{0}^t ||\rho_s - \psi_s||_{L^\infty} ds \leq r \cOne t e^{\gamma \cTwo t} ||\rho - \psi||_{T^*, \gamma}.
\end{equation}
To estimate $|D_2|$, consider two cases. First, suppose
$$\intl{s}^t \Big( h(\rho_\sigma, x) - h(\psi_\sigma x) \Big) d\sigma \geq 0.$$
Then
\begin{eqnarray*}
|D_2| &\leq & \Big| R_2(\rho_s, x) \exp\Big[ - \intl{s}^t \Big( h(\rho_\sigma, x) - h(\psi_\sigma, x) \Big) d\sigma \Big] \\
 &-& R_2(\psi_s, x) \exp\Big[ - \intl{s}^t \Big( h(\rho_\sigma, x) - h(\psi_\sigma, x) \Big) d\sigma \Big]  \Big| \\
 &+& \Big|  R_2(\psi_s, x) \exp\Big[ - \intl{s}^t \Big( h(\rho_\sigma, x) - h(\psi_\sigma, x) \Big) d\sigma \Big]  - R_2(\psi_s, x) \Big|  \\
 &\leq & \Big| R_2(\rho_s, x) - R_2(\psi_s, x)  \Big| \\
 & +&  R_2(\psi_s, x)  \Big\{1 - \exp\Big[ - \intl{s}^t \Big( h(\rho_\sigma, x) - h(\psi_\sigma, x) \Big) d\sigma \Big]   \Big\}.
\end{eqnarray*}
In the other case, when $\intl{s}^t \Big( h(\rho_\sigma, x) - h(\psi_\sigma, x) \Big) d\sigma < 0$, we have analogously
\begin{eqnarray*}
|D_2| &\leq & \Big| R_2(\rho_s, x) - R_2(\psi_s, x)  \Big| \\
&+&  R_2(\rho_s, x)  \Big\{1 - \exp\Big[ - \intl{s}^t \Big( h(\psi_\sigma, x) - h(\rho_\sigma, x) \Big) d\sigma \Big]   \Big\}.
\end{eqnarray*}
Note that both $R_2(\rho_s, x)$ and $R_2(\psi_s, x)$, as both belong to $B_{T^*, \gamma}(r)$, undergo the same estimation (cf (\ref{LemmaOneEq}))
$$ R_2(\rho_s,x), R_2(\psi_s, x) \leq \frac{3}{2} \cOne e^{2 \gamma \cTwo s} r^2 + 2 \cTwo e^{\gamma \cTwo s} r,$$
which allows us to write
\begin{eqnarray}\label{LemmaTwoEstimateOne}
|D_2| &\leq & \Big| R_2(\rho_s, x) - R_2(\psi_s, x)  \Big| + \Big( \frac{3}{2} \cOne e^{2 \gamma \cTwo s} r^2 + 2 \cTwo e^{\gamma \cTwo s} r  \Big) \nonumber \\
& \times &  \Big\{1 - \exp\Big[ - \intl{s}^t \Big| h(\psi_\sigma, x) - h(\rho_\sigma, x) \Big| d\sigma \Big]   \Big\}.
\end{eqnarray}
We have
\begin{eqnarray*}
&& 1 - \exp\Big[ - \intl{s}^t \Big| h(\psi_\sigma, x) - h(\rho_\sigma, x) \Big| d\sigma \Big] \leq \intl{s}^t \Big| h(\psi_\sigma, x) - h(\rho_\sigma, x) \Big| d\sigma \\
&& = \frac{1}{2} \intl{s}^t \Big| \intl{(\Rd)^2} \Big( c_1(x,y;z) + c_1(y,x;z) \Big) \Big( \rho_\sigma(y) - \psi_\sigma(y) \Big) dy dz  \Big| d\sigma \\
&& \leq \intl{s}^t \cOne ||\rho_\sigma - \psi_\sigma||_{L^\infty} d\sigma \leq \intl{s}^t \cOne e^{\gamma \cTwo \sigma} ||\rho - \psi||_{T^*, \gamma} d\sigma\\
&& \leq \ \cOne e^{\gamma \cTwo t} (t-s) ||\rho - \psi||_{T^*, \gamma},
\end{eqnarray*}
which yields
\begin{equation}\label{LemmaTwoEstimateTwo}
1 - \exp\Big[ - \intl{s}^t \Big| h(\psi_\sigma, x) - h(\rho_\sigma, x) \Big| d\sigma \Big] \leq \ \cOne t e^{\gamma \cTwo t} ||\rho - \psi||_{T^*, \gamma}
\end{equation}
Let us estimate
\begin{eqnarray*}
&& \Big| R_2(\rho_s, x) - R_2(\psi_s, x)  \Big| \\
&& \leq \frac{1}{2} \int\limits_{(\mathds{R}^d)^2} c_1(y,z;x) \ \Big| \exp \Big( - \intRd \phi_1(x-u) \rho_s(u) du \Big) \rho_s(y) \rho_s(z) \\
&& \quad - \exp \Big( - \intRd \phi_1(x-u) \psi_s(u) du \Big) \psi_s(y) \psi_s(z) \Big| dy dz \\
&&  + \frac{1}{2} \int\limits_{(\mathds{R}^d)^2} \Big( c_1(x,y;z) + c_1(y,x;z) \Big)\  \Big| \Big[1 - \exp \Big( - \intRd \phi_1(z-u) \rho_s(u) du \Big) \Big]  \\
&&  \quad \times \rho_s(x) \rho_s(y) - \Big[1 - \exp \Big( - \intRd \phi_1(z-u) \psi_s(u) du \Big) \Big] \psi_s(x) \psi_s(y) \Big| dy dz \\
&&  + \intRd  c_2(y;x) \ \Big| \exp \Big( - \intRd \phi_2(x-u) \rho_s(u) du \Big) \rho_s(y) \\
&& \quad - \exp \Big( - \intRd \phi_2(x-u) \psi_s(u) du \Big) \psi_s(y)  \Big| dy \\
&&  + \intRd  c_2(x;y) \ \Big| \Big[1 - \exp \Big( - \intRd \phi_2(y-u) \rho_s(u) du \Big) \Big] \rho_s(x) \\
&&  \quad - \Big[1 - \exp \Big( - \intRd \phi_2(y-u) \psi_s(u) du \Big) \Big] \psi_s(x)  \Big| dy.
\end{eqnarray*}
Denote by $I_i$ the $i$-th component of the RHS of the above inequality for $i=1,2,3,4$. Then estimating analogously as above we derive
$$I_3, I_4 \leq \ \cTwo \Big(e^{2 \gamma \cTwo s} \phiTwo r + e^{\gamma \cTwo s}\Big) \Big|\Big| \rho - \psi \Big|\Big|_{T^*, \gamma}.$$
Moreover, noting that
\begin{eqnarray*}
&&\Big| \rho_s(y)\rho_s(z) - \psi_s(y)\psi_s(z) \Big| \\
&&\leq \frac{1}{2} \Big( \rho_s(z) + \psi_s(z)  \Big) \Big| \rho_s(y) - \psi_s(y)  \Big| + \frac{1}{2} \Big( \rho_s(y) + \psi_s(y)  \Big) \Big| \rho_s(z) - \psi_s(z)  \Big|,
\end{eqnarray*}
we obtain
\begin{eqnarray*}
&& I_1 \leq \frac{1}{2} \cOne \Big( 2 e^{2 \gamma \cTwo s} r + e^{3 \gamma \cTwo s} \phiOne r^2  \Big) \Big|\Big| \rho - \psi \Big|\Big|_{T^*, \gamma}, \\
&&I_2 \leq \ \cOne \Big( 2 e^{2 \gamma \cTwo s} r + e^{3 \gamma \cTwo s} \phiOne r^2  \Big) \Big|\Big| \rho - \psi \Big|\Big|_{T^*, \gamma}.
\end{eqnarray*}
Therefore
\begin{eqnarray}\label{LemmaTwoEstimateThree}
&&\Big| R_2(\rho_s, x) - R_2(\psi_s, x)  \Big|  \leq \Big[ \frac{3}{2} \cOne e^{\gamma \cTwo s} \Big( 2 e^{\gamma \cTwo s} r + e^{2 \gamma \cTwo s} \phiOne r^2  \Big) \nonumber \\
 &&+ 2 \cTwo e^{\gamma \cTwo s} \Big(e^{\gamma \cTwo s} \phiTwo r + 1 \Big)  \Big]  \Big|\Big| \rho - \psi \Big|\Big|_{T^*, \gamma}.
\end{eqnarray}
Substituting (\ref{LemmaTwoEstimateTwo}) and (\ref{LemmaTwoEstimateThree}) into (\ref{LemmaTwoEstimateOne}) and using it together with (\ref{LemmaTwoEstimateZero}), we obtain (cf \ref{LemmaTwoEq})
$$ \Big| \Big(F(\rho) - F(\psi)\Big)_t (x) \Big| \leq e^{\gamma \cTwo t} f(t) \Big|\Big| \rho - \psi \Big|\Big|_{T^*, \gamma}, $$
where
\begin{eqnarray*}
f(t) &=& t \ \Big[ \frac{3}{2} r^2 \cOne e^{2 \gamma \cTwo t} \Big( \cOne t + \phiOne \Big) \\
&+& r e^{\gamma \cTwo t} \Big( 2 \cOne \cTwo t + 3 \cOne + 2 \cTwo \phiTwo  \Big) + 2 \cTwo \Big].
\end{eqnarray*}
Therefore
$$ \Big|\Big| F(\rho) - F(\psi) \Big|\Big|_{T^*, \gamma} \leq \sup\limits_{t \in [0, T^*]} f(t) \Big|\Big| \rho - \psi \Big|\Big|_{T^*, \gamma}  $$
Note that $f(t)$ is continuous, increasing function of $t$ and $f(0) = 0$. Thus, there exists $T^{**} > 0$ such that $f(T^{**}) < 1 $ and $f(t) \in [0, f(T^{**})]$ for $t \in [0,T^{**}]$. Choosing $T^* = \min(T^{**}, \tilde{T})$, we obtain
$$ \Big|\Big| F(\rho) - F(\psi) \Big|\Big|_{T^*, \gamma} \leq C \Big|\Big| \rho - \psi \Big|\Big|_{T^*, \gamma} $$
with $C = f(T^*) \leq f(T^{**}) < 1$.
\end{proof}

\paragraph{\bf
Acknowledgement}
This work has been performed in the framework of the joint Polish-German project No
57154469
 ``Dynamics of Large Systems of Interacting Entities" supported by DAAD. The author gratefully acknowledges the support and
the hospitality extended to him during the stay at Bielefeld University where the main part of the work
was done. Special thanks the author wishes to express to Professor Yuri Kozitsky for suggesting the problem and for many fruitful conversations, as without his help this paper would not come into being.

\bibliographystyle{plain}

\end{document}